\documentclass[english, a4, 11pt]{amsart}

\usepackage{multicol} 
\usepackage{babel}
\usepackage{amstext}
\usepackage{amsthm}
\usepackage{amsmath}
\usepackage{amsfonts}
\usepackage{latexsym}
\usepackage{ifthen}
\usepackage{xypic}
\usepackage{amssymb}

\xyoption{all}
\newcommand{\Rien}[1]{}
\newcommand{\Formel}[1]{(\ref{#1})}
\newcommand{\Ref}[1]{(\ref{#1})}

\newcommand{\BN}{\mathbb B}

\newcommand{\Lem}[1]{(\ref{#1})}
\newcommand{\Prop}[1]{(\ref{#1})}
\newcommand{\Cor}[1]{(\ref{#1})}

\newcommand{\Theo}[1]{(\ref{#1})}

\renewcommand{\O}{{\mathcal O}}

\newcommand{\PN}{{\mathbb P}}

\newcommand{\Pic}{{\rm Pic}}

\newcommand{\lra}{\longrightarrow}
\newcommand{\KC}{{\mathbb C}}
\newcommand{\KZ}{{\mathbb Z}}
\newcommand{\KQ}{{\mathbb Q}}

\newcommand{\KN}{{\mathbb N}}
\newcommand{\End}{{\rm End}}

\newcommand{\Sl}{{\rm Sl}}
\newcommand{\PSl}{{\rm PSl}}

\newcommand{\Sieg}{\mathfrak H}
\newcommand{\OH}{\mathbb H}
\newcommand{\KR}{\mathbb R}
\newcommand{\sO}{{\mathcal O}}

\newtheorem{lemma1}[equation]{}

\newenvironment{lemma}{\begin{lemma1}{\bf Lemma.}}{\end{lemma1}}

\newenvironment{example}{\begin{lemma1}{\bf Example.}\rm}{\end{lemma1}}
\newenvironment{abs}{\begin{lemma1}\rm}{\end{lemma1}}
\newenvironment{theorem}{\begin{lemma1}{\bf Theorem.}}{\end{lemma1}}

\newenvironment{proposition}{\begin{lemma1}{\bf Proposition.}}{\end{lemma1}}
\newenvironment{corollary}{\begin{lemma1}{\bf Corollary.}}{\end{lemma1}}

\newenvironment{rem}{\begin{lemma1}{\bf Remark.}\rm}{\end{lemma1}}
\newenvironment{definition}{\begin{lemma1}{\bf Definition.}}{\end{lemma1}}

\begin{document}

\begin{abstract}
Projective structures on compact real manifolds are classical objects in real differential geometry. Complex manifolds with a holomorphic projective structure on the other hand form a special class as soon as the dimension is greater than one. In the K\"ahler Einstein case $\PN_m$, tori and ball quotients are essentially the only examples. They can be described purely in terms of Chern class conditions. We give a complete classification of all projective manifolds carrying a projective structure. The only additional examples are modular abelian families over quaternionic Shimura curves. They can also be described purely in terms of Chern class conditions.
\end{abstract}
\subjclass{53B10; 14D06; 14K10}
\title{Projective uniformization, extremal Chern classes and quaternionic Shimura curves}
\author[P. Jahnke]{Priska Jahnke}
\address{Priska Jahnke - Fakult\"at T1 - Hochschule Heilbronn - Max-Planck-Stra{\ss}e 39 - D-74081 Heilbronn, Germany}
\email{priska.jahnke@hs-heilbronn.de}
\author[I. Radloff]{Ivo Radloff}
\address{Ivo Radloff - Mathematisches Institut - Universit\"at T\"ubingen - Auf der Morgenstelle 10 - D-72076 T\"ubingen, Germany}
\email{ivo.radloff@uni-tuebingen.de}
\date{\today}
\maketitle

\section*{Introduction}
 The uniformization theorem of compact Riemann surfaces $M$ says the universal covering space $\tilde{M}$ admits an equivariant embedding into $\PN_1(\KC)$, i.e., we may view $\tilde{M}$ as a submanifold of $\PN_1(\KC)$ such that $\pi_1(M)$ becomes a subgroup of $\PSl_1(\KC)$. In higher dimensions this is no longer true. Manifolds $M_m$, $m>1$, with projective uniformization, i.e., the analogous property of $\tilde{M}_m$ admitting an equivariant embedding into $\PN_m(\KC)$, form a very special class first studied by Kobayashi and Ochiai.

One of their results is the following generalization of the case of compact Riemann surfaces to higher dimensions (\cite{KO}): 
\begin{theorem} \label{KE}
On a compact K\"ahler Einstein manifold $M$, the following conditions are equivalent:
\begin{enumerate}
  \item $M$ carries a holomorphic normal projective connection.
  \item The Chen Ogiue inequality $2(m+1)c_2(M) \le mc_1^2(M)$ is an equality.
  \item $M$ is either $\PN_m(\KC)$, a finite {\'e}tale quotient of a torus or a ball quotient.
\end{enumerate}
In any of these cases, $M$ admits a flat holomorphic normal projective connection.
\end{theorem}
\noindent A {\em holomorphic normal projective connection} is an infinitesimal version of projective uniformization $\tilde{M}_m \hookrightarrow \PN(\KC)_m$. A {\em ball quotient} means $\tilde{M}_m \simeq \BN_m(\KC) =\{z \in \KC^m | |z| < 1\}$. Note that if we identify 
$\BN_m(\KC) \subset \KC^m \subset \PN_m(\KC)$ in a standard way, then every affine transformation of $\KC^m$ and in fact every automorphism of $\BN_m(\KC)$ is induced by a projective transformation of $\PSl_m(\KC)$. This explains why the manifolds in 3.) are examples.

Kobayashi and Ochiai's result immediately raises the question of what are the {\em non} K\"ahler--Einstein manifolds $M_m$ with projective uniformization. The assumption is analytic in nature and therefore particularly interesting for {\em algebraic projective} $M_m$, i.e., manifolds that can be described as common zero set of a number of homogeneous algebraic equations in some projective space $\PN_N(\KC)$. 

Another result of Kobayashi and Ochiai says there are no examples in the surface case (\cite{KO}): a projective surface with projective uniformization is K\"ahler Einstein. Before the authors' previous work (\cite{JRproj}) it was in fact unknown whether there exist any non--K\"ahler Einstein examples at all. In \cite{JRproj} the complete classification in the projective threefold case was given and it was proved that there exists exactly one source of non K\"ahler Einstein examples in three dimensions: modular families of false elliptic curves over a Shimura curve (see \Ref{ExmplIntro}). 

In this article the classification is extended to arbitrary dimensions in the flat case:

\begin{theorem} \label{main} On a projective manifold $M_m$ which is \underline{not} K\"ahler Einstein, the following conditions are equivalent:
\begin{enumerate}
  \item $M$ carries a flat holomorphic normal projective connection.
 \item Up to a finite {\'e}tale covering, $M$ admits an abelian group scheme structure $f: M \lra C$ over a compact Shimura curve $C$ such that the Arakelov inequality $2\deg f_*\Omega^1_{M/C} \le (m-1)\deg K_{C} = (m-1)(2g_C-2)$ (\cite{Fa}) is an equality.
 \item Up to a finite {\'e}tale covering, $M \simeq Z \times_C Z \times_C \cdots \times_C Z$, where $Z \lra C$ is a Kuga fiber space constructed from the rational corestriction $Cor_{F/\KQ}(A)$ of a division quaternion algebra $A$ defined over a totally real number field $F$ such that
   \[A \otimes_{\KQ} \KR \simeq M_2(\KR) \oplus {\mathbb H} \oplus \cdots \oplus {\mathbb H}.\]

Here ${\mathbb H}$ denotes the Hamiltonian quaternions (see section~\ref{Ex}).
\end{enumerate}
\end{theorem}
Consequently, if $M_m$ is a projective manifold whose universal covering space can be embedded equivariantly into $\PN_m(\KC)$, then $M_m$ is either K\"ahler Einstein as in \Ref{KE} or a modular abelian fibration over a Riemann surface as in \Ref{main} (and $\tilde{M} \simeq \KC^{m-1} \times \Sieg_1$).

By a result of Viehweg and Zuo, 2.) implies 3.) (\cite{VZ}). The construction of the family $Z \lra C$ in 3.) is a generalization of a construction of modular families of abelian varieties due to Mumford (\cite{Mum}). The curve $C$ can be embedded as a totally geodesic rigid submanifold of some moduli space of abelian varieties. 

Even though the examples are well known, they have apparently never before been considered as examples of manifolds carrying a projective structure. From the point of view of hermitian symmetric spaces, these additional examples have to do with $I_{1,1} \simeq III_1$, i.e., the one dimensional ball is a Siegel space. The flatness assumption in 1.) is necessary for some geometric arguments in section~\ref{birklass}. 

For the proof of \Theo{main} see sections~\ref{AbSch} and \ref{Ex}. The construction of the examples in 3.) is explained in detail in section~\ref{Ex}. We only give some details here and the simplest case in \Ref{ExmplIntro}. The rational corestriction $Cor_{F/\KQ}(A)$ is a central simple $\KQ$--algebra of dimension $4^d$, where $d := [F:\KQ]$, and either
  \[Cor_{F/\KQ}(A) \simeq M_{2^d}(\KQ) \quad \mbox{or} \quad Cor_{F/\KQ}(A) \simeq M_{2^{d-1}}(B)\]
 for some division quaternion algebra $B/\KQ$. We refer to the first case as '$B$ splits'. From this setting one constructs (\cite{Mum}, \cite{VZ}) a modular family $f: Z \lra C$ of abelian varieties over some compact Shimura curve $C$ with simple general fiber $Z_{\tau}$ satisfying
\begin{itemize}
  \item $\dim Z_\tau = 2^{d-1}$ and $End_{\KQ}(Z_{\tau}) \simeq \KQ$ in the case $B$ split,
  \item $\dim Z_\tau = 2^d$ and $End_{\KQ}(Z_{\tau}) \simeq B$ in the case $B$ non--split.
\end{itemize}
The smallest possible dimension is obtained for $F = \KQ$, $A = B$ an indefinite division quaternion algebra ($B$ is necessarily non-split). Here $f: Z \lra C$ is a (PEL--type) family of abelian surfaces, $End_{\KQ}(Z_{\tau}) \simeq B$ for $\tau$ general. Such abelian surfaces are called false elliptic curves, it is the additional three dimensional example found in \cite{JRproj}.

The construction of this family is recalled in 2.) below. Also included in 1.) is the case of modular families of elliptic curves. These are non--compact examples and of no furhter interest for our purposes, but they explain the general idea and show why such families arise in this context.

\begin{example} \label{ExmplIntro} 1.) {\em Modular families of elliptic curves}. Let $\Gamma \subset Sl_2(\KZ)$ be torsion free, for example some congruence subgroup. Then $\Gamma$ acts without fixed points on $\Sieg_1$ as a group of moebius transformations, $C = \Gamma \backslash \Sieg_1$ is a Riemann surface with projective uniformization by construction. It is not possible to choose $\Gamma$ such that $\Gamma \backslash \Sieg_1$ is compact.

Let $\Lambda = \KZ \oplus \KZ$ and denote by $\Gamma_{\Lambda} \simeq \Lambda \rtimes \Gamma$ the group of matrices
   \[\left(\begin{array}{ccc}
                1 & m & n \\
                0 & a & b \\
                0 & c & d 
          \end{array}\right) \subset Sl_3(\KR), \;\; (m,n) \in \Lambda, \;\; \gamma = \left(\begin{array}{cc}
                                                        a & b \\
                                                        c & d
                                                     \end{array}\right) \in \Gamma\]
Think of $\KC \times \Sieg_1$ as an open set of $\PN_2$ by identifying $(z, \tau)$ and $[z:\tau:1]$. The group $\Gamma_{\Lambda}$ acts on $\KC \times \Sieg_1$ projectively by
  \[(z, \tau) \mapsto \left(\frac{z+m\tau+n}{c\tau+d}, \frac{a\tau+b}{c\tau+d}\right),\]
again without fixed points. The quotient $S = \Gamma_{\Lambda} \backslash \KC \times \Sieg_1$ is a smooth quasi projective surface with projective uniformization by construction. It comes with a smooth and proper elliptic fibration $f: S \lra C = \Gamma \backslash \Sieg_1$, the fiber over $[\tau]$ is the elliptic curve $\simeq \KC/(\KZ\tau + \KZ)$. 

However, $S$ will non--compact and every smooth compactification will add rational curves to $S$ in the new fibers and this will destroy the projective uniformization \Ref{ratcurve}. 

 2.) {\em Modular families of false elliptic curves} (\cite{Shim}). A false elliptic curve is an abelian surface $A$ whose $\KQ$--Endormorphismring is a totally indefinite quaternion algebra over $\KQ$. Modular families of such surfaces can be constructed as follows.

Start with a totally indefinite quaternion algebra $B$ over $\KQ$, for example the algebra $\subset M_{2\times 2}(\KR)$ generated over $\KQ$ by
  \[x = \left(\begin{array}{cc}
                              \sqrt{2} & 0 \\
                                 0 & -\sqrt{2}
                            \end{array}\right), \quad y = \left(\begin{array}{cc}
 0 & -3 \\
 1 & 0
\end{array}\right)\]
Then $x^2 = 2$, $y^2=-3$ (omitting $1_{2\times 2}$), $xy=-yx$ and $B = \KQ + \KQ x + \KQ x + \KQ xy$. Let $\Lambda \simeq \KZ^{\oplus 4}$ be some maximal order of $B$, for example (see \cite{Ver})
  \[\Lambda = \KZ + \mbox{$\frac{1}{2}$}\KZ(x+xy) +  \mbox{$\frac{1}{2}$}\KZ(1+y) + \KZ xy.\]
Let $\Gamma$ be a torsion free subgroup of $\Lambda_1^{\times}$, the norm one unit subgroup of $\Lambda$. Then $\Gamma \subset \Sl_2(\KR)$ acts without fixed points on $\Sieg_1$. Contrary to the case of elliptic curves, $\Gamma$ may be chosen such that $C=\Gamma \backslash \Sieg_1$ is compact.

Denote by $\Gamma_{\Lambda} \simeq \Lambda \rtimes \Gamma$ the group of matrices
   \[\left(\begin{array}{cc}
                1_{2\times 2} & \lambda \\
                0_{2\times 2} & \gamma
          \end{array}\right) \subset Sl_4(\KR), \;\; \lambda \in \Lambda, \;\; \gamma \in \Gamma.\]
Think of $\KC^2\times \Sieg_1$ as an open subset of $\PN_3$ by identifying $(z, \tau)$ and $[z:\tau:1]$, $z=(z_1, z_2)$. As in the case of elliptic curves, the group $\Gamma_{\Lambda}$ acts on $\KC^2 \times \Sieg_1$ projectively. The quotient $T = \Gamma_{\Lambda} \backslash \KC^2 \times \Sieg_1$ is a smooth compact threefold whose projective uniformization directly follows from the construction. It comes with a smooth and proper false elliptic curve fibration $f: T \lra C$, the fiber over $[\tau]$ is $A_{\tau} \simeq \KC^2/\Gamma{\tau \choose 1}$. For $\tau$ general ${\rm End}_{\KQ}(A_\tau)\simeq B$. In fact $T$ is projective \Ref{FEK}.
\end{example}

We should mention that there is a notion of an $S$--structure for an arbitrary hermitian symmetric space $S$ of the compact type. It is a useful tool for example in connection with rigidity questions (\cite{KObook}, \cite{HM2}). The rank one case $S = \PN_m(\KC)$ we consider here is somewhat special in this context.

\

\noindent {\bf Acknowledgements.} The authors want to thank N. Nakayama for valuable explanations concerning abelian fibrations as well as E. Viehweg and K. Zuo for explanations concerning their work.

\

\noindent {\bf Notations.} We consider complex K\"ahler manifolds $M$, maps are holomorphic or rational. $T_M$ denotes the holomorphic tangent bundle, $\Omega_M^1$ the dual bundle of holomorphic one forms. We do not distinguish between line bundles and divisors on $M$. $K_M = \det \Omega_M^1$ denotes the canonical divisor.

\section{Holomorphic normal projective connections} \setcounter{equation}{0}
Following classical notation (\cite{Gu}, \cite{KO}) we study manifolds with a holomorphic projective structure or connection. We assume $M$ compact K\"ahler, $m = \dim M$.

\begin{abs}{\bf Holomorphic projective structures.}
Let $V \simeq \KC^{m+1}$ and $\PN_m = \PN(V)$. The group $Sl_{m+1}(\KC)$ acts on $\PN_m$ with kernel $Z(Sl_{m+1}(\KC)) = \mu_{m+1} = \{\lambda \cdot Id_{m+1} | \lambda^{m+1} = 1\}$. We obtain a finite covering map of degree $m+1$
  \[\Sl_{m+1}(\KC) \lra Aut(\PN(V)) = \PSl_{m}(\KC).\]
In fact $\Sl_{m+1}(\KC)$ is the universal covering space of $\PSl_{m}(\KC)$. For later use note that $Z(Sl_{m+1}(\KC))$ is also the kernel of the $m+1$--th symmetric power map.

\begin{definition} \label{DefHS}
 $M$ admits a {\em holomorphic projective structure}, if there exists an atlas $\{(U_i, \varphi_i)\}_{i \in I}$ with holomorphic maps $\varphi_i: U_i \hookrightarrow \PN_m(\KC)$ such that
  \[\varphi_i \circ \varphi_j^{-1}: \varphi_j(U_{ij}) \lra \varphi_i(U_{ij})\]
is the restriction of some $g_{ij} \in \PSl_m(\KC)$ whenever $U_{ij} = U_i \cap U_j \not= \emptyset$.
\end{definition}
We will say {\em $M$ admits a projective structure} for short. The $g_{ij}$ define a class $\xi \in H^1(M, \PSl_m(\KC))$. Let $G_{ij} \in Sl_{m+1}(\KC)$ be a set of matrices representing $g_{ij}$. The $G_{ij}$ do not necessarily define a class in 
 $H^1(M, Sl_{m+1}(\KC))$ as $\xi$ does not necessarily lift in
   \[H^1(M, Sl_{m+1}(\KC)) \lra H^1(M, \PSl_m(\KC)) \lra H^2(M, \mu_{m+1}).\]
The obstruction space $H^2(M, \mu_{m+1})$ is non empty for $M$ projective. An manifold with projective uniformization as in the introduction clearly admits a projective strucutre.

\begin{example} \label{Exmpl} 1.) K\"ahler examples include $\PN_m(\KC)$, finite \'etale quotients of tory, ball quotients (see the introduction). Note that this is the list of K\"ahler manifolds of constant holomorphic sectional curvature. Note also that $\BN_m(\KC)$ is the non compact dual of $\PN_m(\KC)$ in the sense of hermitian symmetric spaces. 

2.) We find more examples if we drop the assumption $M$ projective/K\"ahler: certain Hopf manifolds or twistor spaces over conformally flat Riemannian
fourfolds. 

3.) If $M$ carries a projective structure, then so does every \'etale covering of $M$.
\end{example}
\end{abs}

\begin{abs}{\bf Development.}\label{univcov}(\cite{KO}, \cite{KoWu}) Let $M$ be as in \Ref{DefHS}. Denote by $\tilde{M}$ the universal covering space of $M$. Fixing a starting point $p \in \tilde{M}$, we find an immersive map $\psi: \tilde{M} \to \PN_m(\KC)$ and a map $\rho: \pi_1(M, \pi(p))\simeq \pi_1(M) \to PGl_m(\KC)$ such that $\psi(\gamma(p)) = \rho(\gamma)(\psi(p))$ for any $\gamma \in \pi_1(M)$. The map $\psi$ is called a {\em development}. It is obtained as follows:

A point on $\tilde{M}$ corresponds to a pair $[q, \sigma]$ where $q \in M$ and $\sigma$ is a path from $\pi(p)$ to $q$ in $M$. Cover the trace of $\sigma$ by open subsets $(U_i, \varphi_i)$, $i=1, \dots, r$, as in \Ref{DefHS}, starting at $\pi(p)\in U_1$ and ending at $q \in U_r$. Put $\psi([q, \sigma])=(g_{1,2} \circ \cdots \circ g_{r-1,r})(\varphi_1(\pi(p))$. In the case $q=\pi(p)$ one has 
$\sigma =: \gamma \in \pi_1(M, \pi(p))$ and one defines $\rho(\gamma)=(g_{1,2} \circ \cdots \circ g_{r-1,r})$.

If conversely the universal covering space $\tilde{M}$ of some complex manifold $M$ admits an immersive equivariant map $\tilde{M} \lra \PN_m(\KC)$, then $M$ carries a projective structure.
\end{abs}

\begin{example}
If $M$ is simply connected and $M$ carries a projective structure, then $M \simeq \PN_m(\KC)$. 
\end{example}

\begin{abs}{\bf Induced bundles.}
Let $M$ be as in \ref{DefHS}. Let $E$ be a holomorphic $\PSl_m(\KC)$--homogeneous vector bundle on $\PN(V)$, i.e., a vector bundle with the property that the action of 
$\PSl_m(\KC)$ on $\PN(V)$ extends to the total space $E$. Then $E$ induces a holomorphic vector bundle on $M$. Indeed, the local pieces $\varphi_i^*E$ defined on $U_i \subset M$ glue. 

This is not clear for $\Sl_{m+1}(\KC)$--homogeneous bundles, i.e bundles where the action of $\Sl_{m+1}(\KC)$ on $\PN(V)$ extends. As a prototype example consider
   \[\O_{\PN(V)}(1) = -\frac{K_{\PN(V)}}{m+1}.\]
This line bundle is $Sl_{m+1}(\KC)$ homogeneous but not $\PSl_m(\KC)$ homogeneous. In other words, $\O_{\PN(V)}(1)$ comes from the representation of the parabolic subgroup of $\Sl_{m+1}(\KC)$
  \begin{equation} \label{parabP}
      P = \{\left(\begin{array}{cc}
              A & 0 \\
              b^t & a
           \end{array}\right) | A \in Gl_m(\KC), b \in \KC^m, a \in \KC^*, \det A \cdot a = 1\}
  \end{equation}
given by projection onto $a$. But this representation is not induced by a representation of $\PN(P) \subset \PSl_m(\KC)$. 

If $E$ is a  $Sl_{m+1}(\KC)$--homogeneous bundle and if the representation of $P$ commutes with $S^{m+1}$, then $S^{m+1}E$ is a $\PSl_m(\KC)$ homogeneous bundle. It is the case in the above example: $S^{m+1}\O_{\PN(V)}(1)\simeq K_{\PN(V)}$ is $\PSl_m(\KC)$ homogeneous, the induced bundle on $M$ is $-K_M$. The obstruction of extracting the $m+1$--th root of $K_M$ in $\Pic(M)$ again lies in $H^2(M, \mu_{m+1})$.

\begin{lemma} \label{SurjFlat}
   On every manifold $M$ with a projective structure we have a surjection
   \begin{equation}
      B \lra -K_M \lra 0
   \end{equation}
   where $B$ is some rank $r>0$ flat vector bundle coming from a representation $\pi_1(M) \lra Sl_r(\KC)$.
\end{lemma}

\begin{proof}
In \Ref{parabP}, projection onto $A$ or $a$  gives two standard representations of $P$. As a $P$ module, $V$ sits in an extension $0 \lra K \lra V \lra L \lra 0$ where $K$ and $L$ are complex vector spaces of dimension $\dim K = m$, $\dim L = 1$. The induced sequence of $\Sl_{m+1}(\KC)$ homogeneous bundles on $\PN(V)$ is, up to dualizing and twist, the Euler sequence or first jet sequence
  \begin{equation} \label{EulSeq}
    0 \lra \Omega^1_{\PN(V)} \otimes \O_{\PN(V)}(1) \lra V \otimes \O_{\PN(V)} \lra \O_{\PN(V)}(1) \lra 0.
  \end{equation}
\Rien{The $P$--module structure on $V$ comes from the natural $Sl_{m+1}(\KC)$--structure, implying that the corresponding homogeneous vector bundle on $\PN(V)$ is trivial.}
The $m+1$--th symmetric power gives a surjection of homogeneous $\PSl_m(\KC)$ bundles $S^{m+1}(V \otimes \O_{\PN(V)}) \lra -K_{\PN(V)} = \O_{\PN(V)}(m+1)$. On $M$ we obtain a surjection $B \lra -K_M$ where $B$ is flat. Of course $B$ is just the bundle induced by $S^{m+1}G_{ij}$ for a choice of $G_{ij}$ as in \Ref{DefHS}
\end{proof}
\end{abs}

There is an infinitesimal description of projective structures:

\begin{abs}{\bf Holomorphic normal projective connections.}
Recall the definition of the {\em Atiyah class} (\cite{At}): Associated to a holomorphic vector bundle $E$ on the complex manifold $M$ one has the first jet sequence
  \[0 \lra \Omega_M^1 \otimes E \lra J_1(E) \lra E \lra 0.\]
The obstruction to the holomorphic splitting is a class $b(E) \in H^1(M,$ $Hom(E,E)\otimes \Omega_M^1)$ called {\em Atiyah class of $E$}. For properties of $b(E)$ see \cite{At}. We only mention that if $\Theta^{1,1}$ denotes the $(1,1)$--part of the curvature tensor of some differentiable connection on $E$, then, under the Dolbeault isomorphism, $b(E)$ corresponds to $[\Theta^{1,1}] \in H^{1,1}(M, Hom(E,E))$. In particular, as $M$ is K\"ahler, $tr(b(E)) = -2i\pi c_1(E) \in H^1(M, \Omega_M^1)$. This is why we normalise and put $a(E) := -\frac{1}{2i\pi}b(E)$.

\begin{definition}
$M_m$ carries a {\em holomorphic normal projective connection} if the (normalised) Atiyah class of the holomorphic cotangent bundle
has the form
  \begin{equation} \label{AtProj} 
a(\Omega_M^1) = \frac{c_1(K_M)}{m+1} \otimes id_{\Omega_M^1} + id_{\Omega_M^1}
  \otimes \frac{c_1(K_M)}{m+1} \in H^1(M, \Omega_M^1 \otimes T_M \otimes
  \Omega_M^1),
  \end{equation}
where we use $\Omega_M^1 \otimes T_M \otimes
\Omega_M^1 \simeq\Omega_M^1 \otimes
\End(\Omega_M^1) \simeq \End(\Omega_M^1) \otimes \Omega_M^1$. 
\end{definition}
We will say {\em $M$ has a projective connection} for short. It was shown in \cite{MM} how a holomorphic cocycle solution to
\Formel{AtProj} can be thought of as a $\KC$--bilinear holomorphic connection map
  $\Pi: T_M \times T_M \to T_M$
satisfying certain rules modelled after the Schwarzian derivative. Conversely, the existence of such a connection implies \Formel{AtProj}.  We will not use this description.

\begin{proposition}
 If $M$ admits a projective structure, then $M$ admits a projective connection. The connection is called flat in this case. 
\end{proposition}

For a proof see \cite{KO}. It is not known whether all projective manifolds admitting a projective connection also admit a flat projective connection. $M$ admits a projective structure if and only if $\Pi = 0$ is a cocycle solution to \Formel{AtProj}. $M$ carries a flat projective connection iff $M$ admits a projective structure.
\end{abs}

\begin{abs}{\bf Chern Classes.}
Let $M_m$ be as above compact K\"ahler with a projective connection. As in the case of projectice space one has (\cite{KO}):
\begin{equation} \label{chern}
  c_r(M) = \frac{1}{(m+1)^r}{m+1 \choose r}c_1^r(M), \; r=0, \dots, m \; \quad \mbox{in }
H^r(M, \Omega_M^r).
 \end{equation}
In particular, $2(m+1)c_2(M) = mc_1^2(M)$.
\end{abs}

\section{K\"ahler Einstein Manifolds} \label{KEKlass} \setcounter{equation}{0}
In the K\"ahler Einstein case one has \Theo{KE} Theorem from the introduction. We include the proof for the convenience of the reader: The Chen--Ogiue inequality (\cite{COE}) says that on $(M_m, \omega)$ compact K\"ahler--Einstein
  \begin{equation} \label{CO}
    \int_M (mc_1^2(M) - 2(m+1)c_2(M)) \wedge \omega^{m-2} \le 0,
  \end{equation}
with equality if and only if $(M, \omega)$ is of constant holomorphic sectional curvature. A compact complex manifold $M_m$ carries a K\"ahler metric of constant holomorphic sectional curvature $s$ if and only if
  \begin{itemize}
   \item $M \simeq \PN_m(\KC)$ (case $s>0$) or
   \item $M$ is an {\'e}tale quotient of a torus (case $s=0$) or
   \item $M$ is a ball quotient (case $s < 0$).
  \end{itemize}
This is the equivalence of 2.) and 3.) in \Theo{KE} Theorem. In the case of 1.) we have \Formel{chern}, implying equality in \Formel{CO}. Hence 1.) implies 2.). By example~\ref{Exmpl} the third point 3.) implies 1.). The proof of \Theo{KE} is complete.

\section{Projective Manifolds} \label{birklass}\setcounter{equation}{0}
Let $M_m$ be a projective manifold with a projective connection, not necessarily flat. Recall that a line bundle $L$ is called {\em nef}, if $L.C \ge 0$ for every irreducible curve $C$ in $M$. A vector bundle $E$ is {\em nef}, if $\O_{\PN(E)}(1)$ is nef. A nef line bundle is {\em big} if $L^m>0$.

The manifold is called {\em minimal} if $K_M$ is nef. It is called a {\em good minimal model}, if $|dK_M|$ is spanned for $d \gg 0$. In this case one has the {\em Iitaka fibration}
   \begin{equation} \label{IitakafibXX}
      f: M \lra Y.
   \end{equation}
Here $Y$ is normal, fibers are connected. One of the central conjectures in birational algebraic geometry predicts: $M$ minimal implies $M$ is a good minimal model. The conjecture is proven in dimension $\le 3$ but open in higher dimension.

\begin{abs}{\bf Rational Curves.} We say that $M$ contains a rational curve, if there exists a non constant holomorphic map $\nu: \PN_1(\KC) \lra M$.

\begin{proposition} \label{ratcurve}
  Let $M_m$ be a projective manifold with a projective connection. If $M$ contains a rational curve, then $M \simeq \PN_m(\KC)$.
\end{proposition}

\noindent We include here a proof of the flat case, for the general case see \cite{JRproj}.

\begin{proof} Let $\psi: \tilde{M} \lra \PN_m(\KC)$ be a development of the universal covering space $\mu: \tilde{M} \lra M$. We have an induced map $\tilde{\nu}: \PN_1(\KC) \lra \tilde{M}$ such that $\mu \circ \tilde{\nu} = \nu$. Then $\nu^*T_M = (\psi \circ \tilde{\nu})^*T_{\PN_m(\KC)}$ is ample. This forces $M \simeq \PN_m(\KC)$ by Mori's proof of Hartshorne's conjecture (\cite{Mori}, in particular \cite{MP} I, Theorem 4.2.).
\end{proof}

\begin{rem}
1.) The assumption $M$ projective (K\"ahler) is necessary as soon as $\dim M > 2$ (\cite{KO}). Indeed, the twistor space $Tw(S)$ over a compact conformally flat real fourfold $S$ is a compact complex threefold with a projective structure. Any such $Tw(S)$ is differentiably covered by rational curves. There are choices for $S$ such that $Tw(S) \not= \PN_3(\KC)$. But then $Tw(S)$ is not K\"ahler (\cite{Hitch}). 

2.) \Prop{ratcurve} holds mutatis mutandis for any projective manifold $M$ with a flat $S$--structure, $S$ an irreducible hermitian symmetric space of the compact type: on any rational curve, $T_M$ is nef by the above argument, implying $M$ uniruled. In the case of ${\rm rank} S> 1$ we have $M \simeq S$ by \cite{HM}. \end{rem}

\begin{corollary} \label{corratcurve}
Let $M_m \not\simeq \PN_m(\KC)$ be a projective manifold with a projective connection. Then
 \begin{enumerate}
  \item $K_M$ is nef and $K_F$ is nef for every smooth submanifold $F \subset M$.
  \item  any rational map $\xymatrix{M' \ar@{..>}[r] & M}$ from a complex manifold $M'$ must be holomorphic.
\end{enumerate}
\end{corollary}

\begin{proof} As $M_m \not\simeq \PN_m(\KC)$, $M_m$ is free of rational curve by \Prop{ratcurve}. Then 1.) follows from the cone theorem \cite{KM}.

 2.) if we had to blow up $M'$ in order to make the map holomorphic, then $M$ would contain a rational curve.
\end{proof}

Combining \Ref{KE} and \Ref{ratcurve} and deep results of Aubin and Yau we find:

\begin{corollary} \label{KEC}
  Let $M_m$ be a projective manifold with a projective connection
   \begin{enumerate}
     \item if $K_M$ is big and nef, then $M$ is a ball quotient.
     \item if $K_M \equiv 0$, then $M$ is a finite \'etale quotient of a torus.
   \end{enumerate}
\end{corollary}

\begin{proof}
   If $K_M \equiv 0$, then $M$ admits a K\"ahler-Einstein metric \cite{Yau}. We coclude by \Ref{KE}.

If $K_M$ is big and nef, then $|dK_M|$ is spanned for $d \gg 0$ by the base point free theorem \cite{KM}. The Iitaka fibration $f: M \lra Y$ is birational, exceptional fibers are covered by rational curves by \cite{Kaw}, Theorem 2. By \Ref{ratcurve} $M$ is free of rationale curves, so $f$ is an isomorphism. Then $K_M$ is ample. Then $M$ is K\"ahler-Einstein \cite{Au}. We again conclude by \Ref{KE}.
\end{proof}
\end{abs}

\begin{abs}{\bf Abelian group schemes.} A fibration $f: M \lra Y$ is an {\em abelian group scheme}, if $Y$ is smooth, $f$ is submersive, every fiber of $f$ is a smooth abelian variety, and $f$ admits a smooth section. Our aim is to prove:

\begin{theorem} \label{AbSc} Let $M_m \not\simeq \PN_m(\KC)$ carry a projective structure. Then, up to a finite \'etale cover, $M$ is an abelian group scheme over a base $N$ of general type.
\end{theorem}

The proof can be found at the end of this section. First note some special cases: the example of $M$ a finite \'etale quotient of an abelian variety is the case $\dim N = 0$. The case of $M$ a ball quotient is contained in the case $\dim N = \dim M$.

The proof of \Theo{AbSc} uses a result of Koll\'ar. Recall from  (\cite{Ko}, 1.7.) that $M$ is said to have {\em (generically) large fundamental group} if for any irreducible complex subvariety $Z \subset M$ of positive dimension (passing through a general point)
  \[{\rm Im}[\pi_1(Z_{norm}) \lra \pi_1(M)] \quad \mbox{is infinite.}\]
Here $Z_{norm}$ denotes the normalization of $Z$.

\begin{proposition} \label{lfg}
  Any $M \not\simeq \PN_m(\KC)$ with a projective structure has large fundamental group.
\end{proposition}

\begin{proof} By \ref{corratcurve}, $K_M$ is nef. We prove by contradiction. Assume ${\rm Im}[\pi_1(Z_{norm}) \to \pi_1(M)]$ is finite for some $Z$ as above. Denote the normalization map by $\nu: Z_{norm} \to Z$.

Let $C_{norm}$ be some general curve in $Z_{norm}$, i.e., the intersection of $\dim Z - 1$ general hyperplane sections. Think of $C_{norm}$ as the normalization of $C := \nu(C_{norm}) \subset Z$. We have
  \[\pi_1(C_{norm}) \lra \pi_1(Z_{norm}) \lra \pi_1(M).\]
Then ${\rm Im}[\pi_1(C_{norm}) \lra \pi_1(M)]$ is finite. The kernel of $\pi_1(C_{norm}) \lra \pi_1(M)$ induces a finite {\'e}tale covering $C' \lra C_{norm}$ from a compact Riemann surface $C'$, such that $\mu: C' \lra C \subset M$ factors over $\tilde{M}$, the universal covering space of $M$. Let $\psi: \tilde{M} \lra  \PN_m(\KC)$ be a development. Denote the induced map $C' \lra \tilde{M} \lra \PN_m(\KC)$ by $\psi_1$. Then $\mu^*T_M = \psi_1^*T_{\PN_m(\KC)}$ is ample, implying $-K_M.C > 0$. This contradicts $K_M$ nef.
\end{proof}

\begin{proposition} \label{abundance}
  For any $M \not\simeq \PN_m(\KC)$ with a projective structure there exists a finite \'etale cover $M' \lra M$ such that $M'$ is a good minimal model (see \Ref{IitakafibXX}) and the Iitaka fibration $f: M' \lra Y$ satisfies
  \begin{enumerate}
  \item $Y$ is birational to a smooth $W$ of general type.
  \item a general fiber of $f$ is a finite \'etale quotient of an abelian variety.
  \end{enumerate}
\end{proposition}

\begin{proof}
By \ref{corratcurve}, $K_M$ is nef. We have to prove $|dK_M|$ is spanned for $d \gg 0$. We may replace $M$ by finite \'etale covers. By \ref{SurjFlat} we have a surjection
\begin{equation} \label{SurjFlat2}  
    B \lra -K_M \lra 0
\end{equation}
from a flat bundle $B$ of some rank $r$ which is induced by a representation 
   \[\rho: \pi_1(M) \lra \Sl_{r}(\KC)\]
For a holomorphic $h: Z \lra M$, we have the map $\pi_1(Z) \lra \pi_1(M)$ and composition gives a representation of $\pi_1(Z)$. We call it the ``pull back of $\rho$ to $Z$''. It will be denoted $h^*\rho$.

{\em 1. Case. $\rho$ has a finite image.} Then there exists a finite \'etale covering $\nu: \tilde{M} \lra M$ such that $\nu^*B$ is trivial. The pull back of \Formel{SurjFlat2} shows that $-K_{\tilde{M}} \simeq \nu^*(-K_M)$ is globally generated. On the other hand $K_{\tilde{M}}$ is nef. Then $K_{\tilde{M}} \simeq \O_{\tilde{M}}$.

{\em 2. Case. $\rho$ has an infinite image.}  Denote by $G$ the component of the identity of the Zariski closure of $\rho(\pi_1(M))$. After a finite \'etale cover we may assume $\rho(\pi_1(M)) \subset G$. Let $Rad(G)$ be the solvable radical of $G$. 

Assume $G$ is not solvable. Then $G/Rad(G)$ is semisimple. Consider $\bar{\rho}: \pi_1(M) \lra G/Rad(G)$ obtained by composition. The image is Zariski dense. By \cite{Ko}, 3.5., there exists $g: \xymatrix{M \ar@{..>}[r] & W}$ dominant to some smooth $W$ of dimension $\dim W > 0$ with the following property: if $Z \subset M$ passes through a very general point, then $Z$ will be contracted by $g$ if and only if
  \[\bar{\rho}(Im(\pi_1(Z^{norm}) \lra \pi_1(M))) \subset G/Rad(G)\]
is finite. Here $W$ is a smooth model of $Sh_{kern \bar{\rho}}(M)$.  We denote by $\pi: M_1 \lra M$ a resolution such that $g_1: M_1 \lra W$ is holomorphic. 

After replacing $M$ by a finite \'etale cover we may instead of finiteness assume $\pi^*\bar{\rho} = g_1^*\sigma$ for some {\em big} $\sigma: \pi_1(W) \lra G/Rad(G)$ (\cite{Zuo}, explanations after Theorem 1). We claim that $W$ is of general type. In the case $G/Rad(G)$ almost simple this is \cite{Zuo}, Theorem 1. The general case easily follows by induction on the number of almost simple almost direct factors of $G/Rad(G)$ and $C_{n,m}$ from \cite{Kaw85}. 

The map $g$ is almost holomorphic by  \cite{Ko}, 4.1., i.e., there exist Zariski open dense subsets $M^0 \subset M$ and $W^0 \subset W$ such that $M^0 \lra W^0$ is proper. Denote a general proper fiber of $g$ by $F$.

\begin{lemma} \label{almab}
  $F$ is an \'etale quotient of an abelian variety.
\end{lemma}

The Lemma also applies to the case $G = Rad(G)$ solvable. It proves in this case that $M$ is an \'etale quotient of an abelian variety.

\begin{proof}[Proof of \Ref{almab}]
Denote the inclusion map by $\iota: F \hookrightarrow M$. The pull back $\iota^*\bar{\rho}$ is trivial so $\iota^*\rho: \pi_1(F) \lra Rad(G)$. We find a basis of our representation space such that $\rho(\iota_*\gamma)$ has upper triangular form for any $\gamma \in \pi_1(F)$. This gives a filtration
 \[0 \subset E_1 \subset E_2 \subset \cdots \subset E_r = \iota^*B,\]
into flat bundles $E_i$ of rank $i$ on $F$. Each quotient $E_{i+1}/E_i$ is a flat line bundle on $F$, i.e., is induced by a representation of $\pi_1(F)$. 

Flat line bundles are nef (this is not true for vector bundles of $rk > 1$). Extensions of nef bundles are nef. This shows $\iota^*B$ is nef. Quotients of nef vector bundles bundles are nef. The pull back of \Formel{SurjFlat2} shows $\iota^*(-K_M)$ is nef. By adjunction, using the fact that $g$ is almost holomorphic, $\iota^*(-K_M) \simeq -K_F$. On the other hand $K_F$ is nef by \Ref{corratcurve}. This shows $K_F \equiv 0$.

By the Bogomolov--Beauville decomposition theorem (\cite{Bo}, \cite{Be}) there exists an \'etale covering 
   \[A \times Z \lra F\]
where $A$ is abelian and $Z$ simply connected. By \Prop{lfg}, $\pi_1(Z) = \{id\}$ implies $Z$ is a point. Therefore, $F$ is covered by an abelian variety. This completes the proof of \Ref{almab}.
\end{proof}

We continue the proof of \Ref{abundance}. By \cite{Kaw85}, $C_{n,m}$ is true for $g_1$, so
  \[\kappa(M) = \kappa(M_1) \ge \kappa(F) + \kappa(W) = \dim W.\]
Hence, if $f: \xymatrix{M \ar@{..>}[r] & Y}$ denotes the rational Iitaka fibration, then $\dim Y = \kappa(M) \ge \dim W$. On the other hand \Ref{almab} implies that $f$ contracts the general fiber of $g$. This gives a rational dominant map $\xymatrix{W \ar@{..>}[r] & Y}$. Then $Y$ and $W$ are birational.
 
By \cite{Lai}, 4.4., $M$ has a good minimal model. By \cite{Lai}, 2.5., $M$ is a good minimal model, as $K_M$ is nef. This completes the proof of \Ref{abundance}.
\end{proof}

\begin{proof}[Proof of \Ref{AbSc}]
  Replace $M$ by the finite \'etale cover from \Ref{abundance}, denote by $f: M \lra Y$ the Iitaka fibration whose general fibers are finite \'etale quotients of abelian varieties. As $M$ is not K\"ahler Einstein, $0 < \dim Y < \dim M$ by \Ref{KEC}.

By \Ref{lfg}, $M$ has large fundamental group. By \cite{Ko}, 6.3.~Theorem, some finite \'etale cover of $M$ is birational to an abelian group scheme $\alpha: A \lra S$. Replace $M$ by this cover. We arrived at the diagram
  \begin{equation} \label{diagr}
   \xymatrix{A \ar[d]^{\alpha} \ar@{..>}[r]^h & M \ar[d]^f \\
              S  \ar@{..>}[r]^g & Y.}
  \end{equation}
By \Ref{corratcurve}, $h$ is holomorphic. The map $\alpha$ has a smooth section $s$. Then $g = f \circ h \circ s: S \lra Y$ must be holomorphic, too.  By \cite{Kaw}, Theorem 2, $f$ is equidimensional. Let $F$ be some fiber of $f$ over $q \in Y$ with induced scheme structure. The map $h$ is finite, generically $1:1$ on every fiber $A_p$ of $\alpha$. Diagram \Formel{diagr} shows $F$ is irreducible and reduced (consider local sections of $\alpha$).  For $p \in g^{-1}(q)$ we may think of $h_p = h|_{A_p} : A_p \lra F_{red}$ as the normalization of $F$.

Let $Y' = h(s(S))$. Any positive dimensional fiber of $g: S \lra Y$ is covered by some fiber of $h: A \lra M$. Fibers of a birational morphism between smooth manifolds are rationally chain connected. As $M$ does not contain any rational curve, $h \circ s$ must contract every fiber of $S \lra Y$ to a point. Then $Y' \simeq Y$. Now $Y'.F=1$ implies $Y' \simeq Y$ smooth. By \Ref{abundance}, $Y$ is of general type.

Since $Y$ and $M$ are both smooth and $f$ is flat, the relative dualizing sheaf $\omega_{M/Y} = \omega_M \otimes f^*\omega_Y^{-1}$ is locally free, hence every fiber $F$ of $f$ is Gorenstein with trivial canonical bundle $K_F \simeq \sO_F$. Then $K_{A_p} = h_p^*K_F + {\mathcal N}$ where ${\mathcal N}$ is the conductor ideal. As $K_{A_p}$ and $K_F$ are both trivial, ${\mathcal N}$ must be trivial. Then $F$ is normal and $h_p$ is an isomorphism. Then $M \lra N:=Y$ is an abelian group scheme. This completes the proof of \Ref{AbSc}.
\end{proof}

\begin{example} \label{ExpNeu}
1.) Let $\Gamma \subset Sl_2(\KR)$ be torsion free such that $C = \Gamma \backslash \Sieg_1$ is a smooth Riemann surface. Then $a(\gamma, \tau) := c\tau+d$ is a well defined factor of automorphy on $\Sieg_1$. It defines a theta characteristic which we denote $\frac{K_C}{2}$. The standard representation $\rho: \Gamma \lra Sl_2(\KR)$ induces a flat bundle $E$ on $C$. It comes with an extension 
  \[0 \lra \frac{K_C}{2} \lra E \lra -\frac{K_C}{2} \lra 0,\]
the first jet sequence of $-\frac{K_C}{2}$. $E$ is flat but not nef. Here $G = Sl_2(\KC)$ is almost simple.

2.) Consider the elliptic curve example $S = \Gamma_{\Lambda} \backslash \KC \times \Sieg_1$ from \Ref{ExmplIntro}, 1.) with proper elliptic fibration $f: S \lra C$. Here $K_S \simeq 3f^*(\frac{K_C}{2})$ where $\frac{K_C}{2} \in \Pic(C)$ denotes the theta characteristic as in 1.).

The standard representation $\rho: \Gamma_{\Lambda} \lra \Sl_3(\KC)$ induces a flat bundle $E$ on $S$. It comes with an extension 
  \[0 \lra \Omega^1_S \otimes f^*(-\frac{K_C}{2}) \lra E \lra f^*(-\frac{K_C}{2}) \lra 0.\]
It is the first jet sequence of $-\frac{K_S}{3} = f^*(-\frac{K_C}{2})$ as in the proof of \Ref{SurjFlat}. Restricted to a fiber of $S \lra C$, $E$ is given by the pull back representation 
 \[(m, n) \mapsto \left(\begin{array}{ccc}
                1 & m & n \\
                0 & 1 & 0 \\
                0 & 0 & 1 
          \end{array}\right)\]
and we obtain the filtration as in the proof of \Ref{almab}. If $F$ denotes the non--splitting extension of $\O$ and $\O$ on the elliptic curve, then $E$ is the non--splitting extension of $\O$ and $F$. Using the notation from the proof of \Ref{abundance}, $Rad(G) \simeq (\KC^2, +)$ is non trivial while $G/Rad(G) \simeq \Sl_2(\KC)$ is almost simple. 
\end{example}
\end{abs}

\section{Abelian group schemes} \label{AbSch}
\setcounter{equation}{0}
Because of \Ref{AbSc} we consider from now on an abelian group scheme $f: M_m \lra N_n$ where $M$ carries a holomorphic normal projective connection. We have the exact sequence
   \begin{equation} \label{reltang}
    0 \lra f^*\Omega_N^1 \stackrel{df}{\lra} \Omega_M^1 \lra
    \Omega^1_{M/N} \lra 0
   \end{equation}
of holomorphic forms and
\[E = E^{1,0} = f_*\Omega_{M/N}^1\]
is a holomorphic rank $m-n$ vector bundle on $N$ such that $f^*E \simeq
\Omega_{M/N}^1$ via the canonical map $f^*f_*\Omega_{M/N}^1 \to
\Omega_{M/N}^1$. Our aim is to prove

\begin{theorem} \label{VZklass}
  In the above situation, assume in addition that $M$ is not K\"ahler--Einstein. Then $N$ is a compact Riemann surface and 
   \[2\deg E^{1,0} =  (m-1) \deg(K_N),\]
  i.e., the Arakelov inequality is an equality. 
\end{theorem}
Recall the Arakelov inequality (\cite{Fa}, \cite{JostZuo}): if $f: M \lra N$ is an abelian group scheme over a compact Riemann surface $N$, then $2\deg E^{1,0} \le (m-1) \deg(K_N)$. The theorem will be proved at the end of this section.

\begin{proposition} \label{fam}
 In the above general situation of an abelian group scheme (K\"ahler Einstein or not), $N$ admits a projective connection, and
    \begin{equation} \label{AtE}
      a\big(E(-\frac{K_N}{n+1})\big) = 0 \quad
      \mbox{in } H^1(N, \End(E) \otimes
      \Omega_N^1),\end{equation}
  where $a$ denotes the normalised Atiyah class.
\end{proposition}

\begin{rem} \label{remat}
1.) The formula is in terms of classes, we do not assume the
existence of a theta characteristic on $N$, i.e., a line bundle $\frac{K_N}{n+1} \in \Pic(N)$.

2.) In the case $N$ a compact Riemann surface, \Formel{AtE} implies $E \simeq U \otimes \frac{K_N}{2}$ for some flat bundle $U$ coming from a representation of $\pi_1(N)$ (\cite{At}) and some theta characteristic $\frac{K_N}{2} \in \Pic(N)$. 
\end{rem}

Some consequences first. The
trace of the (normalised) Atiyah class gives the first Chern class, hence
  \begin{equation} \label{relK}
    c_1(E) = \frac{m-n}{n+1}c_1(K_N) \quad \mbox{in } H^1(N,
    \Omega_N^1).
  \end{equation}
Let as usual $K_{M/N} = K_M - f^*K_N$. Then $K_{M/N} = \det \Omega_{M/N}^1 = f^*\det E$. We may rewrite \Formel{relK} as follows:
\begin{corollary}\label{ChernEq}
  In the situation of \Prop{fam} the following identities hold in  $H^1(M, \Omega_M^1)$:
    \[c_1(K_{M/N}) = \frac{m-n}{n+1}c_1(f^*K_N) \quad \mbox{and} \quad
    c_1(K_M) = \frac{m+1}{n+1} c_1(f^*K_N).\]
  In particular, $c_1(K_M)$ and $c_1(f^*K_N)$ are proportional.
\end{corollary}

\begin{corollary}
\begin{enumerate}
 \item 
If $K_M \equiv 0$ or $K_N \equiv 0$, then $M$ and $N$ are {\'e}tale quotients of abelian varieties.
 \item If $K_M$ or $K_N$ are not nef, then $M \simeq N \simeq \PN_m(\KC)$.
\end{enumerate}
\end{corollary}

\begin{proof}
  1.) By \Ref{ChernEq}, if $K_M \equiv 0$ or $K_N \equiv 0$, then $K_M \equiv 0$ and $K_N \equiv 0$. The claim follows from \Ref{KEC} and \Ref{fam}.

2.) If $K_M$ or $K_N$ are not nef, then $K_M$ and $K_N$ are not
nef by \Ref{ChernEq}. Then $M \simeq \PN_m(\KC)$ and $N \simeq \PN_n(\KC)$ by \Ref{ratcurve}. As $n>0$ by assumption, $m = n$ and
$f$ is an automorphism of projective space.
\end{proof}

\begin{proof}[Proof of \Prop{fam}]
  Denote the section of $f$ by $s: N \to M$. Consider
  the pull back to $N$ by $s$ of \Ref{reltang}:
   \begin{equation} \label{reltang2}
    0 \lra \Omega_N^1 \stackrel{s^*df}{\lra} s^*\Omega_M^1 \lra
    s^*\Omega^1_{M/N} \simeq E \lra 0.
   \end{equation}
  We have the map $ds: s^*\Omega_M^1 \lra \Omega^1_N$.   
 As $(ds)(s^*df) = d(f \circ s) = id_{\Omega_N^1}$, \Ref{reltang2}
 splits holomorphically. 

The normalised Atiyah class of
  $s^*\Omega_M^1$ as a bundle on $N$ is obtained from $a(\Omega^1_M)$ by
  applying $ds$ to the last $\Omega_M^1$ factor in
  \Formel{AtProj}. We obtain
    \begin{equation} \label{resat} 
    a(s^*\Omega_M^1) = \frac{s^*c_1(K_M)}{m+1} \otimes ds + id_{s^*\Omega_M^1} \otimes
    \frac{c_1(s^*K_M)}{m+1}
\mbox{ in } H^1(N, s^*\Omega_M^1 \otimes s^*T_M \otimes \Omega_N^1),
\end{equation} where
 we carefully distinguish between $s^*c_1(K_M) \in H^1(N,
 s^*\Omega_M^1)$ and the class $c_1(s^*K_M) = ds(s^* c_1(K_M)) \in H^1(N, \Omega_N^1)$.

 The Atiyah class of a direct sum is the direct sum of the Atiyah
  classes (\cite{At}). As the pull back of \Formel{reltang} splits
  holomorphically, we get the Atiyah classes of $\Omega_N^1$ and $E$ by
  projecting \Formel{resat} onto the corresponding summands. 

We first compute $a(E)$. The class $c_1(K_M) \in H^1(M,
\Omega_M^1)$ is the pull back of some class in $H^1(N,
\Omega_N^1)$; it therefore vanishes under $H^1$ of $\Omega_M^1 \to
\Omega^1_{M/N}$. This means the first summand in \Formel{resat}
vanishes if we project, while the second summand becomes
  \begin{equation} \label{AtEeins}
    id_E \otimes \frac{c_1(s^*K_M)}{m+1}  \;\; \mbox{in } H^1(N, E
  \otimes E^* \otimes
    \Omega_N^1).\end{equation}
This is $a(E)$. The trace gives
  \[c_1(E) = rk(E)  \frac{c_1(s^*K_M)}{m+1} \in H^1(N,
  \Omega_N^1).\]
The determinant of
\Formel{reltang2} combined with the last formula and $rk E = m-n$ gives
  \begin{equation} \label{AdFor}
     c_1(K_N) = c_1(s^*K_M) - c_1(E) =
     \frac{n+1}{m+1}c_1(s^*K_M)
  \end{equation}
in $H^1(N, \Omega_N^1)$. Now \Formel{AtE} follows from \Formel{AtEeins}.

Next we compute $a(\Omega_N^1)$. We first have to apply $ds$ to the
first factor of the first summand in \Formel{resat}. This gives
$\frac{c_1(s^*K_M)}{m+1}$. Contracting the middle factor using $s^*df$ we obtain
  \begin{equation} \label{Nhnpc}
    \frac{c_1(s^*K_M)}{m+1} \otimes id_{\Omega_N^1} + id_{\Omega_N^1} \otimes
    \frac{c_1(s^*K_M)}{m+1}  \in H^1(N, \Omega_N^1 \otimes T_N \otimes
    \Omega_N^1).
  \end{equation}
This is $a(\Omega_N^1)$. After replacing $c_1(s^*K_M)$ with the formula found in \Formel{AdFor}, we see that $N$ admits a projective connection (compare \Formel{AtProj}). This comples the proof of \Ref{fam}.
\end{proof}

\begin{proof}[Proof of \Ref{VZklass}]
By \Ref{AbSc}, the Iitaka fibration $f: M \lra N$ is an abelian group scheme where $N$ is of general type. As $M$ is not K\"ahler Einstein, $0 < \dim N < \dim M$ by \Ref{KEC}. By \Ref{fam}, $N$ has a projective connection. By \Ref{KEC}, $N$ is a ball quotient, so $K_N$ is ample.
 
For any torsion free sheaf ${\mathcal F}$ on $N$ of positive rank and ample $H\in Pic(N)$ the $H$--slope is defined $\mu_H({\mathcal F}) = \frac{c_1({\mathcal F}).H^{n-1}}{{\rm rk}{\mathcal F}}$ as usual. Let ${\mathbb V} := R^1f_*\KC$. By \cite{VZ2}, Theorem~1 and Remark~2, we have $\mu_{K_N}({\mathbb V}) = 2\mu_{K_N}(E) \le \mu_{K_N}(\Omega_N^1)$. By \Cor{ChernEq}
  \[2\mu_{K_N}(E) = \frac{2}{n+1}c_1(K_N)^n \le \mu_{K_N}(\Omega_N^1) = \frac{1}{n}c_1(K_N)^n.\]
Since $c_1(K_N)^n > 0$ we find $n = 1$. Then $N$ is Riemann surface and we have in fact equality $2\mu_{K_N}(E) = \mu_{K_N}(\Omega_N^1)$, i.e., the Arakelov inequality is an equality.
\end{proof}

We summarize our results in the first part of the proof of the main theorem \Ref{main}:

\begin{proof}[Proof of 1.) $\Rightarrow$ 2.) $\Rightarrow$ 3.) in \Theo{main}]
From \Ref{AbSc} and \Ref{VZklass} we obtain 1.) $\Rightarrow$ 2.). By \cite{VZ}, Theorem 0.5., 2.) $\Rightarrow$ 3.).
\end{proof}

It remains to show that the examples in 3.), \Theo{main} indeed carry a projective structure. Before we come to this we add some explanations for the convenience of the reader:
 
\begin{rem} \label{VHS} Let $f:M \lra N$ be an abelian group scheme over a compact Rieman surface $N$ with a projective connection. The push forward of \Formel{reltang} gives the Kodaira Spencer map
   \begin{equation} \label{KodSp}
     E = f_*\Omega_{M/N}^1 \lra K_N \otimes R^1f_*\O_M.
  \end{equation}
It is an isomorphism in the present case (\cite{VZ}, Proposition~1.2.). As a consequence of Simpson's correspondence, $E$ is poly stable, i.e., a direct sum of stable bundles of the same slope (\cite{VZ}, Proposition 1.2.). Kobayashi Hitchin correspondence and \Prop{fam} implies $E \simeq U \otimes \frac{K_N}{2}$ for some unitary flat bundle $U$ and some theta characteristic $\frac{K_N}{2} \in \Pic(N)$ on $N$ (compare remark~\ref{remat}). The isomorphism \Formel{KodSp} gives $R^1f_*\O_M \simeq  U \otimes (-\frac{K_N}{2})$.

There exists a lift of $\pi_1(N)$ to a torsion free subgroup $\Gamma \subset Sl_2(\KR)$, s.t. $R^1f_*\KC$ comes from the tensor product of the canonical representation of $\Gamma$ and some unitary representation (\cite{VZ}, Proposition 1.4., Lemma 4.1.). We may assume that the representation induces $U$. Then
  \[0 \lra E \lra R^1f_*\KC \otimes \O_N \lra R^1f_*\O_M \lra 0\]
is the unique non--splitting extension of $-\frac{K_N}{2}$ by $\frac{K_N}{2}$ tensorized by $U$. In fact one has a tensor product description of $R^1f_*\KQ$ defined over a number field. That $\Gamma$ comes from a quaternion algebra as in 3.) of \Theo{main} is in the end a consequence of a result of Takeuchi (\cite{Ta}). The tensor description of $R^1f_*\KQ$ will play a role in the next section.
\end{rem}

\section{Projective non--K\"ahler--Einstein examples} \label{Ex} \setcounter{equation}{0}
The aim of this section is to prove that the explicit examples in 3.) of \Theo{main} carry a flat projective structure. The examples are well known families of abelian varieties (\cite{Shim}, \cite{Mum}, \cite{VZ}, \cite{vG}) and have been studied from many points of view, but it was apparently unknown that they carry such a structure.

\begin{abs}{\bf Data.} \label{Data2013} Let $A$ be a division quaternion algebra defined over some totally real number field $F$ of degree $[F:\KQ] = d$. Assume that $A$ splits at exactly one infinite place, i.e.,
  \begin{equation} \label{decom}
     A \otimes_{\KQ} \KR \simeq M_2(\KR) \oplus \OH \oplus \cdots \oplus \OH.
  \end{equation}
The existence of such $A$'s follows from Hilbert's reciprocity law.
Let $Cor_{F/\KQ}(A)$ be the rational corestriction of $A$. Then
  \[Cor_{F/\KQ}(A) = M_{2^{d-1}}(B),\]
where $B$ is a quaternionen algebra over $\KQ$, possibly split (\Lem{CorB}). From this data we construct in section~\ref{Alg}
  \begin{enumerate}
   \item a torsion free discrete subgroup $\Gamma$ of $Sl_2(\KR)$ acting canonically on $U_{\KR} = \KR^2$ such that $\Gamma \backslash \Sieg_1$ is compact,
   \item an orthogonal representation $\rho: \Gamma \lra O(g)$ on $W_{\KR} \simeq \KR^g$ (where $g = 2^{d-1}$ in the case $B$ split and $g = 2^d$ in the case $B$ non split), such that
   \item the symplectic representation $id \otimes \rho$ fixes some complete lattice in $U_{\KR} \otimes W_{\KR}$ on which the symplectic form only takes integral values.
  \end{enumerate}
We will first explain how the above data leads to an abelian group scheme $M_\Gamma \lra C_\Gamma = \Gamma \backslash \Sieg_1$ with a projective structure. Here $M_\Gamma$ will be compact if and only if $C_\Gamma$ is which is the case if the above data is indeed derived from a division quaternion algebra. For an explicit example see \Ref{ExmplIntro}.

\Rien{\begin{example}\label{ellcurv} Let $\Gamma \subset Sl_2(\KZ)$ be some standard torsion free congruence subgroup. Choose $\rho = id$. 

Consider the quotient $Z$ of $\KC \times \Sieg_1$ by the action
  \[(z, \tau) \mapsto \left(\frac{z+m\tau+n}{c\tau+d}, \frac{a\tau+b}{c\tau+d}\right), \quad \left(\begin{array}{cc}
                      a & b \\
                      c & d
                     \end{array}\right) \in \Gamma, \;\; m,n \in \KZ.\]
Then $Z$ is a smooth manifold with a map $Z \lra C = \Gamma \backslash \Sieg_1$. It is a modular family of elliptic curve. By \Lem{univcov}, $Z$ has a projective structure.
\end{example}}
\end{abs}

\begin{abs}{\bf Construction of the abelian scheme $M \lra C$.} Assume first we have a collection of data 1.) -- 3.). How to derive it from a quaternion algebra will be shown in section~\ref{Alg}.

\vspace{0.2cm}

\noindent {\em (i) Choice of bases.} Point 1.) includes a choice of a basis of $U_{\KR} \simeq \KR^2$. For later use it is more appropriate to choose an arbitrary basis of $W_{\KR}$, not necessarily orthogonal. Write the elements of $U_{\KR}$ and $W_{\KR}$ as vertical vectors in the chosen bases. The action of $\Gamma$ on $U_{\KR}$ is by left multiplication. The standard symplectic form on $U_{\KR}$
\[\langle u, u'\rangle = u^tJ_2u', \quad J_2 := \left(\begin{array}{cc}
                                           0 & 1 \\
                                          -1 & 0
                                         \end{array}\right)\]
identifies $Sl_2(\KR) = Sp_2(\KR)$ and $U_{\KR} \simeq U^*_{\KR}$ as $\Gamma$--modules. Write the elements of $U^*_{\KR}$ in the dual base as horizontal vectors. The action of $\Gamma$ is then given by right multiplication $\gamma(u) = u\gamma^{-1}$, the $\Gamma$--isomorphism $U_{\KR} \simeq U_{\KR}^*$ by $u \mapsto u^tJ_2$.

The choice of a basis for $U_{\KR}$ and $W_{\KR}$ gives an isomorphism
  \[W_\KR \otimes U_{\KR}^* \simeq M_{g \times 2}(\KR).\]
Fix this isomorphism and think of the elements of $W_\KR \otimes U_{\KR}^*$ as real $g \times 2$ matrices from now on. Any $\gamma \in \Gamma$ acts by
   \[\gamma(\alpha) = \rho(\gamma)\alpha\gamma^{-1}, \quad \alpha \in M_{g\times 2}(\KR).\]
Since $W_\KR \otimes U_{\KR}^* \simeq W_\KR \otimes U_{\KR}$ as $\Gamma$--modules, we find a complete lattice $\Lambda \subset M_{g \times 2}(\KR)$ invariant under the action of $\Gamma$ by 3.). The points 2.) and 3.) say we find a $\rho(\Gamma)$ invariant symmetric and positiv definit $S \in M_g(\KR)$, s.t. the induced symplectic form on $M_{g\times 2}(\KR)$ given by
  \begin{equation} \label{E} 
     E(\alpha, \beta) := tr(\alpha^tS\beta J_2)
  \end{equation} 
takes only integral values on $\Lambda$. For later considerations note that $M_{g\times 2}(\KR) \simeq W_\KR \otimes U_{\KR}^*$ is a symplectic $O(S) \times Sl_2(\KR)$ module via $(\delta, \gamma)(\alpha) = \delta\alpha\gamma^{-1}$.

\vspace{0.2cm}

\noindent {\em (ii) Definition of $M$.} Let $\Gamma_{\Lambda}$ be set of matrices
 \[\gamma_{\lambda} := \left(\begin{array}{cc}
      \rho(\gamma) & \rho(\gamma)\lambda \\
        0_{g \times 2} & \gamma
    \end{array}\right) \in Gl_{g+2}(\KR), \quad \gamma \in \Gamma, \lambda \in \Lambda.\]
Then $\Gamma_{\Lambda} \simeq \Lambda \rtimes \Gamma$ and the sequence
   \[0 \lra \Lambda \lra \Gamma_{\Lambda} \lra \Gamma \lra 1\]
given by $\lambda \mapsto id_{\lambda}$ and $\gamma_{\lambda} \mapsto \gamma$ is exact.  The projective action on $\KC^g \times \Sieg_1$, i.e., where $\gamma_{\lambda}\in \Gamma_{\Lambda}$ acts by
  \begin{equation} \label{quot} (z, \tau) \mapsto \left(\frac{\rho(\gamma)(z + \lambda {\tau \choose 1})}{c\tau + d}, \frac{a\tau + b}{c\tau + d}\right), \quad \gamma = \left(\begin{array}{cc}
                           a & b \\
                           c & d
                       \end{array}\right),\end{equation}
is properly discontinously and free, since the action of $\Gamma$ and of each stabilizer group is. The quotient is a smooth complex manifold 
  \[M = M_{\Gamma, \Lambda} := \Gamma_\Lambda \backslash \KC^g \times \Sieg_1.\]
In the remainder of this subsection we prove:
\begin{proposition}
  The manifold $M$ from above is a compact abelian group scheme over the Riemann surface $C_{\Gamma} := \Gamma \backslash \Sieg_1$ and admits a projective structure.
\end{proposition}
$M$ has a projective structure by \Lem{univcov}. By construction there is a natural holomorphic proper submersion $f: M \lra C_{\Gamma}$ with a section given by $[\tau] \mapsto [(0,\tau)]$. The fiber $M_{\tau} = f^{-1}([\tau])$ is isomorphic to $\KC^g$ divided by the corresponding stabilizer subgroup of $\Gamma_{\Lambda}$. Since the action of $\Gamma$ on $\Sieg_1$ is free, 
  \[M_{\tau} \simeq \KC^g/\Lambda_{\tau},\]
where $\Lambda_{\tau}$ is the image of $\Lambda$ under
  \begin{equation} \label{ComplStr}
    M_{g\times 2}(\KR) \simeq W_{\KR} \times U_{\KR}^* \lra \KC^g, \quad \alpha \mapsto \alpha_{\tau} := \alpha\left(\begin{array}{c}\tau \\ 1 \end{array}\right).
  \end{equation}

\noindent {\em (iii) Projectivity of $M_{\tau}$.} \Formel{ComplStr} endows $M_{g\times 2}(\KR) \simeq W_\KR \otimes U_{\KR}^*$ with the complex structure given by
  \[J_{\tau} = \frac{1}{\Im m \tau}\left(\begin{array}{cc}
                           -\Re e \tau & \tau\bar{\tau} \\
                            -1 & \Re e \tau
                        \end{array}\right) \in Sl_2(\KR), \;\;
\mbox{ i.e., }\,\, i \cdot \alpha_{\tau} = (\alpha J_{\tau}^{-1})_{\tau}.\]
If $\tau' = \gamma(\tau)$ for some $\gamma \in Sl_2(\KR)$, then $J_{\tau'} = \gamma J_{\tau}\gamma^{-1}$. Hence $M_\tau$ is projective if $(M_{g\times 2}(\KR) \simeq W_{\KR} \otimes U_{\KR}^*, \Lambda, J_{\tau}, E)$ satisfies the Riemann conditions:

Recall that $M_{g\times 2}(\KR) \simeq W_{\KR} \otimes U_{\KR}^*$ is a symplectic $O(S) \times Sl_2(\KR)$--module. The complex structure is given by $(id, J_\tau) \in O(S) \times Sl_2(\KR)$. Therefore $E(\alpha J_{\tau}^{-1}, \alpha'J_{\tau}^{-1}) = E(\alpha, \alpha')$, i.e., $E$ is compatible with the complex structure. Positivity for $\tau = i$ follows from $J_i = J_2$ and $E(\alpha, \alpha J_2^{-1}) = tr(\alpha^tS\alpha) > 0$ for $\alpha \not= 0$. Every $\tau \in \Sieg_1$ is of the form $\tau = \gamma(i)$ for some $\gamma \in Sl_2(\KR)$. The invariance of $E$ and $J_{\tau} = \gamma J_{i}\gamma^{-1}$ shows $E(\alpha, \alpha' J_{\tau}^{-1})$ is a positive definite symmetric form for any $\tau \in \Sieg_1$. Therefore $E$ is a positive, integral $(1,1)$ form and $M_{\tau}$ is projective.

\vspace{0.2cm}

\noindent {\em (iv) Isomorphic $M_{\tau}$'s.} If $\tau' = \gamma(\tau)$ for some $\gamma \in \Gamma$, then $(M_{\tau}, E) \simeq (M_{\tau'}, E)$, where $\varphi: M_{\tau} \lra M_{\tau'}$, as a map $\KC^g \lra \KC^g$, is given by
  \begin{equation} \label{analyt}
   \frac{1}{c\tau + d}\rho(\gamma).
  \end{equation}
Indeed,
  \[\Lambda_{\tau'} = \Lambda\left(\begin{array}{c}\gamma(\tau) \\ 1 \end{array}\right) = \frac{1}{c\tau+d}(\Lambda\gamma)_\tau = \frac{\rho(\gamma)}{c\tau+d}(\rho(\gamma^{-1})\Lambda\gamma)_\tau = \frac{\rho(\gamma)}{c\tau+d}\Lambda_{\tau}.\]
On the underlying real vector space $M_{g\times 2}(\KR) \simeq W_\KR \otimes U^*_\KR$, $\varphi$ is given by $\alpha \mapsto \rho(\gamma)\alpha\gamma^{-1}$ showing that $\varphi$ preserves the polarization $E$.

The family of $M_{\tau}$'s over $\Sieg_1$, i.e., the quotient $\KC^g \times \Sieg_1$ divided by the action of the subgroup $\Lambda \subset \Gamma_\Lambda$ as in \Formel{quot}, is projective. We obtain $M$ by dividing out $\Gamma_{\Lambda}/\Lambda \simeq \Gamma$. Fiberwise this is nothing but \Formel{analyt} and the polarizations glue. Then $M$ is projective.
\end{abs}

\begin{rem} \label{fasprod} 1.) The quotient of $\Gamma \backslash \KC^g \times \Sieg_g$, where $\gamma \in \Gamma$ acts as $\gamma_0$ (i.e., put $\lambda=0$ in \Formel{quot}) yields the fiberwise universal covering space $\simeq f_*T_{M/C_\Gamma}$ of $M \lra C_\Gamma$ (compare remarks~\ref{remat} and \ref{VHS}). The total space also carries a projective structure. The local system $R^1f_*\KC$ is given by the dual of $id \otimes \rho$ (compare remark~\ref{VHS}).

2.) Given a collection of data, we may replace $\rho$ by $\rho^d = \rho \oplus \cdots \oplus \rho$ and $W_\KR$ by $W_\KR^d$ for an arbitrary $d \in \KN$. This again gives 1.)--3.) as in \Ref{Data2013}. It is clear from the construction that the new collection leads to $M_{\Gamma} \times_{C_\Gamma} M_\Gamma \times_{C_\Gamma} \cdots \times_{C_\Gamma} M_\Gamma$ with $d$ factors (compare \Theo{main}, 3.).  
\end{rem}

\begin{abs}{\bf From $A$ to $Z$.} \label{Alg} It remains to show how a quaternion algebra $A$ as in \Theo{main}, 3.) leads to a collection of data 1.) -- 3.) as in \Ref{Data2013} (see also remark~\ref{fasprod}, 2.)). We first recall some results

\vspace{0.1cm}

\noindent {\em (i) On central simple algebras.} Let $A$ be a central simple algebra of finite dimension over a field $K$. It is called {\em division} if it is a skew field. It is called a {\em quaternion algebra} if $[A:K] = 4$. A quaternion algebra is either division or {\em split}, i.e., $A \simeq M_2(K)$. Let $Br(K)$ be the Brauer group of $K$. The order $e(A)$ of $[A] \in Br(K)$ is finite and is called the {\em exponent} of $A$. A theorem of Wedderburn says $A \simeq M_r(D)$, where $D/K$ is a division algebra. The $K$--dimension $[D:K] = s(D)^2$ for some $s(D)=s(A) \in \KN$ is called {\em (Schur-) index} of $A$. One has $e(A) | s(A)$ and if $K$ is a local or global field, then even $e(A) = s(A)$.

\

\noindent {\em (ii) The construction.} The corestriction $Cor_{F/\KQ}(A)$ is a $4^d$ dimensional central simple $\KQ$--algebra. The corestriction induces a map of Brauer groups
  \begin{equation} \label{BrMap}
     Br(F) \lra Br(\KQ).
  \end{equation}

\begin{lemma} \label{CorB}
  For $A$ as in \Formel{decom}
  \begin{equation} \label{Bdef}
    Cor_{F/\KQ}(A) \simeq M_{2^d}(\KQ) \quad \mbox{or} \quad Cor_{F/\KQ}(A) \simeq M_{2^{d-1}}(B),
  \end{equation}
 where $B$ is a division quaternion algebra over $\KQ$ which is indefinite (i.e., $B \otimes \KR \simeq M_2(\KR)$) iff $d$ is even, definite (i.e., $B \otimes \KR \simeq \OH$) iff $d$ is odd.
\end{lemma}

\begin{proof} The order of $[A]$ in $Br(F)$ is $e(A)=s(A) = 2$. Consider \Formel{BrMap}. We find $e(Cor_{F/\KQ}(A)) = 1$ or $2$, implying \Formel{Bdef}. From $A \otimes_{\KQ} \KR \simeq M_2(\KR) \oplus \OH^{\oplus d-1}$ using $\OH \otimes_{\KR} \OH \simeq M_4(\KR)$ we obtain
 \begin{equation} \label{CorR} Cor_{F/\KQ}(A) \otimes_{\KQ} \KR \simeq M_2(\KR) \otimes_{\KR} \OH^{\otimes_{\KR} d-1} \simeq
 \left\{\begin{array}{ll}
         M_{2^{d}}(\KR), & d \mbox{ odd} \\
         M_{2^{d-1}}(\OH), & d \mbox{ even.}
      \end{array}\right.
 \end{equation}
The statement on $B \otimes \KR$ now follows from combining \Formel{CorR} and \Formel{Bdef}.
\end{proof}

We refer to the first case in \Formel{Bdef} as ``$B$ splits'', because here $Cor_{F/\KQ}(A)$ $\simeq M_{2^{d-1}}(B)$ for $B = M_2(\KQ)$. Consider the following $\KQ$--vector spaces 
 \[V_{\KQ} = \KQ^{2^d}, \mbox{ in the case $B$ split}, \; V_{\KQ} = B^{2^{d-1}}, \mbox{ in the case $B$ non--split.}\]
The $\KQ$--dimension is $2^d$ and $2^{d+1}$, respectively. The elements of $V_{\KQ}$ are considered as vertical vectors and $V_{\KQ}$ as a left $Cor_{F/\KQ}(A)$--module. 

Over the reals we have by \Formel{CorR} $Cor_{F/\KQ}(A) \otimes_{\KQ} \KR \simeq M_2(\KR) \otimes_{\KR} \OH^{\otimes_{\KR} d-1} \simeq$
 \begin{equation} \label{CorR2} 
 \simeq
 \left\{\begin{array}{ll}
         M_2(\KR) \otimes M_{2^{d-1}}(\KR), &  \mbox{$B$ indefinite / $d$ odd} \\
         M_2(\KR) \otimes M_{2^{d-2}}(\OH), &  \mbox{$B$ definite / $d$ even.}
      \end{array}\right.
 \end{equation}
Consider the following left $Cor_{F/\KQ}(A) \otimes_{\KQ} \KR$--modules: $\KR^2 \otimes \KR^{2^{d-1}}$ in the case $B$ split,  $\KR^2 \otimes M_{2^{d-1}, 2}(\KR)$ in the case $B$ non--split indefinite and finally $\KR^2 \otimes \OH^{2^{d-2}}$ in the case $B$ definite. In all three cases, denote the first factor by $U_{\KR} = \KR^2$, the second by $W_{\KR}$. Then $U_{\KR} \otimes W_{\KR} \simeq V_{\KQ} \otimes \KR$ as left $Cor_{F/\KQ}(A) \otimes_{\KQ} \KR$ modules, i.e., the action of  $Cor_{F/\KQ}(A) \otimes_{\KQ} \KR$ on $U_{\KR} \otimes W_{\KR}$ is a real form of the action of $Cor_{F/\KQ}(A)$ on $V_{\KQ}$.

\

The corestriction comes with a map
  \[Nm: A^{\times} \lra Cor_{F/\KQ}(A)^{\times}.\]
Let $x \mapsto x'$ be the canonical involution of $A$ and
  \[G = \{x \in A | xx' = 1\}.\]
Then $G$ is an algebraic group over $\KQ$. Via $Nm$, $G(\KQ)$ acts on $V_{\KQ}$. Let $\Lambda \subset V_{\KQ}$ be some complete lattice, $\Gamma \subset G(\KQ)$ be some torsion free arithmetic subgroup fixing $\Lambda$ via $Nm$.

Consider the action of $G$ on $V_\KQ$ over $\KR$: The elements of norm $1$ in $M_2(\KR)$ and $\OH$ form the groups $Sl_2(\KR)$ and $SU(2)$, respectively. From $A \otimes_{\KQ} \KR \simeq M_2(\KR) \oplus \OH^{\oplus d-1}$ we infer
  \[G(\KR) = Sl_2(\KR) \times \underbrace{SU(2) \times \cdots \times SU(2)}_{d-1}.\]
The isomorphism $\OH \otimes_{\KR} \OH \simeq M_4(\KR)$ induces $SU(2) \times SU(2) \lra SO(4)$. Then $Nm$ factors over
 \begin{equation} \label{GR2013} G(\KR) \lra \left\{\begin{array}{llc}
         Sl_2(\KR) \times SO(2^{d-1}) , & & \mbox{$B$ indefinit / $d$ odd} \\
         Sl_2(\KR) \times SU(2^{d-1}), &  & \mbox{$B$ definit / $d$ even}.
      \end{array}\right.
 \end{equation}
Projection onto the first factor embeds $\Gamma$ into $Sl_2(\KR)$ and gives $U_{\KR}$ from above the usual symplectic structure as a $\Gamma$--module. Projection onto the second factor gives $W_{\KR}$ the structure of an orthogonal $\Gamma$--module (in the case $W_{\KR} = \OH^{2^{d-2}}$ take the real part of the unitary form). From \Ref{GR2013} we see that $\Gamma$ fixes a symplectic form $E$ obtained by tensorising the standard symplectic form on $U_{\KR}$ and some orthogonal form on $W_{\KR}$ such that $E$ only takes integral values on $\Lambda$. This gives 1.) -- 3.) as in \Ref{Data2013}.
\end{abs}

\begin{rem}
The group $G$ can be identified with the special Mumford Tate group of the constructed family of abelian varieties. The ring $End_{\KQ}(Z_{\tau})$ for $\tau$ general is isomorphic to the group of endomorphisms of $V$ commuting with the action of $G$. 

In the case $B$ non split, $V_{\KQ}$ becomes a $B$--module via $\beta v := v\beta'$ and this action clearly commutes. This gives an embedding $B \hookrightarrow End_{\KQ}(Z_{\tau})$. By (\cite{VZ}, Lemma 6.9., here $B$ is not explicitely mentioned) we find $End_{\KQ}(Z_{\tau}) \simeq B$ in the case $B$ non--split and $End_{\KQ}(Z_{\tau}) \simeq \KQ$ in the case $B$ split.
\end{rem}

\begin{example}(False elliptic Curves) \label{FEK}
(\cite{Shim}) We conclude by describing how the families of false elliptic curves already explained in \Ref{ExmplIntro} fit into the general picture in \Ref{Data2013}. 

In the above setting it is the case $d=1$, $F = \KQ$, $A = B$ indefinit. Denote quaternion conjugation in $B$ by $'$. Choose some pure quaternion $y$ (i.e., $y = -y'$) such that $b := y^2 < 0$. There exists $x \in B$ such that $xy=-yx$ and $a := x^2 > 0$. Then $B$ is generated by $x$ and $y$ as an algebra over $\KQ$ and
  \[x \mapsto \left(\begin{array}{cc}
                              \sqrt{a} & 0 \\
                                 0 & -\sqrt{a}
                            \end{array}\right), \quad y \mapsto \left(\begin{array}{cc}
 0 & b \\
 1 & 0
\end{array}\right)\]
gives an emedding $B \hookrightarrow M_{2\times 2}(\KR)$ and a basis of $U_\KR = \KR^2$. If we identify $B$ with its image in $M_{2\times 2}(\KR)$, s.t. reduced norm and trace are given by usual matrix determinant and trace, respectively, then we are in situation at the beginning \Ref{ExmplIntro}.

Here $V_\KQ=B$ and we have to choose a complete lattice $\Lambda$ in $B$, for example a maximal order. For $\Gamma$ we have to choose an arithmetic torsion free subgroup of the norm one unit subgroup $\Lambda^\times_1 \subset \Sl_2(\KR)$. For $\rho: \Gamma \lra O(W_\KR = \KR^2, 2)$ we have to choose the trivial representation, as $d=1$. This gives a collection of data 1.)--3.) as in \Ref{Data2013}. Indeed, identify $W_\KR \otimes U_\KR^* = M_{2}(\KR)$, where $M_2(\KR)$ is the space containing $B$. $\Gamma$ acts on $M_{2}(\KR)$ via $\alpha \mapsto \rho(\gamma)\alpha\gamma^{-1} = \alpha\gamma^{-1}$, fixing the lattice $\Lambda$. In \Ref{ExmplIntro} and in \Ref{Data2013} the same group $\Gamma_{\Lambda} \subset \Sl_4(\KR)$ is defined leading to the same quotient $\Gamma_{\Lambda}\backslash \KC^2 \times \Sieg_1$.

For the positive symmetric form on $W_\KR$ we choose $S_1 := J_2y$. Then $E_1(\alpha, \beta) = tr(\alpha^tS_1\beta J_2)$ is a $\Gamma$--stable symplectic form on $M_2(\KR)$. The extension of quaternion conjugation from $B$ to $M_{2}(\KR)$ is given by $\alpha \mapsto J_2^{-1}\alpha^tJ_2$. Then (recall $tr(\alpha\beta) = tr(\beta\alpha)$ in
$M_2(\KR)$)
  \[E_1(\alpha, \beta) = tr(\alpha^tS\beta J_2) = tr(-J_2\beta^tS\alpha) = tr(-J_2y\alpha J_2\beta^t) =\]\[=  tr(y\alpha J_2^{-1}\beta^tJ_2) = tr(y \alpha\beta').\]
The last description shows that $E_1$ only takes rational values on $B$. Then some multiple of $E_1$ only takes integral values on $\Lambda$ which gives the desired form $E$ (Shimura's form from \cite{Shim}).
\end{example}

\end{document}